\documentclass[11pt,reqno]{amsart}
\usepackage{amsmath,amsfonts,amssymb,amsthm,stmaryrd,enumerate,url}
\usepackage{microtype}
\usepackage{bbm,mathrsfs}
\usepackage[utf8]{inputenc}  %% pour les accents
\usepackage[english]{babel}
\usepackage[numbers,sort]{natbib}
\usepackage[usenames,dvipsnames,svgnames,table]{xcolor}
\usepackage[colorlinks=true, pdfstartview=FitV, pagebackref=true]{hyperref}
\usepackage{enumitem}

\usepackage[a4paper,vmargin=3.5cm,inner=3cm,outer=3cm]{geometry}

\newtheorem{theorem}{Theorem}
\newtheorem{lemma}[theorem]{Lemma}
\newtheorem{proposition}[theorem]{Proposition}

\newtheorem{Ass}{Assumption}

\newtheorem{remark}{Remark}

\newcommand{\R}{\mathbb R}
\newcommand{\N}{\mathbb N}

\renewcommand{\P}{\mathbb P}
\newcommand{\E}{\mathbb E}

\renewcommand{\d}{{\mathrm d}}

\newcommand{\cB}{{\mathcal B}}

\newcommand{\cD}{{\mathcal D}}

\newcommand{\cF}{{\mathcal F}}

\newcommand{\cI}{{\mathcal I}}

\newcommand{\cN}{{\mathcal N}}
\newcommand{\cO}{{\mathcal O}}

\renewcommand{\tilde}{\widetilde}
\renewcommand{\hat}{\widehat}

\newcommand{\nn}{\textsf{nn}}

\DeclareMathOperator{\var}{Var}

\newcommand{\given}{\, \big|\, }
\newcommand{\ind}{{\mathbf{1}}}
\newcommand{\med}{\mathsf{median}}

\begin{document}

\title[On deviation probabilities in non-parametric regression]{On deviation probabilities in non-parametric regression with heavy-tailed noise}

\author[ ]{Anna Ben-Hamou}
\author[ ]{Arnaud Guyader}

\begin{abstract}
This paper is devoted to the problem of determining the deviation bounds that are achievable in non-parametric regression. We consider the setting where features are supported on a bounded subset of $\R^d$, the regression function is Lipschitz, and the noise is only assumed to have a finite second moment. We first specify the fundamental limits of the problem by establishing a general lower bound on deviation probabilities, and then construct explicit estimators that achieve this bound. These estimators are obtained by applying the median-of-means principle to classical local averaging rules in non-parametric regression, including nearest neighbors and kernel procedures. 

\bigskip

  
\end{abstract}

\maketitle

\section{Introduction}

\subsection{Setting and main results}

Let $(X,Y)$ be a pair of random variables, where $X$ has distribution $\mu$ on $\R^d$, for $d\geq 1$, and $Y$ is real-valued and satisfies $\E[Y^2]<\infty$. We denote by $Q$ the distribution of the pair $(X,Y)$. The regression function is defined for $\mu$-almost every $x\in\R^d$ as 
\[
r(x):=\E[Y\given X=x]\, .
\]
The pair $(X,Y)$ can be written as
\[
Y=r(X)+\varepsilon\, ,
\]
where the random variable $\varepsilon$, called the noise, satisfies $\E[\varepsilon\given X]=0$. 
Note that 
\[
\E\left[ (Y-r(X))^2\right]=\inf_g\E\left[(Y-g(X))^2\right]\, ,
\]
where the infimum is taken over all measurable functions $g:\R^d\to\R$ such that $\E[g(X)^2]<\infty$. In other words, the random variable $g(X)$ is an optimal approximation of $Y$ by a square-integrable function of $X$, with respect to the $L_2$ risk. 

When the distribution $Q$ of the pair $(X,Y)$ is unknown, one cannot predict $Y$ using $r(X)$. However, assuming that one has access to an i.i.d.\ sample
\[
\cD_n:=\left((X_1,Y_1),\dots,(X_n,Y_n)\right) 
\]
with the same distribution as $(X,Y)$, one can use the data $\cD_n$ in order to construct an estimate $\hat{r}_n$ of the function $r$. 

Throughout the article, $\R^d$ is equipped with the Euclidean distance, and for $x\in\R^d$ and $\varepsilon>0$, $\cB(x,\varepsilon)$ denotes the Euclidean closed ball of center $x$ and radius $\varepsilon$. We will be interested in the following model (see Section~\ref{sec:related-work} for comments on this set of hypotheses). 

\begin{Ass}\label{kjncllxnlankx}
The class $\cF=\cF_{\rho,\sigma}$, with $\rho,\sigma>0$, is the class of distributions $Q$ for $(X,Y)$ satisfying:
\begin{enumerate}
\item[(i)]\label{item1} The support $S$ of $\mu$ is bounded with diameter $D>0$ and for all $x\in S$ and $\varepsilon\in (0,D]$, we have 
\begin{equation}\label{aljcj}
\mu\left(\cB(x,\varepsilon)\right)\geq \rho\varepsilon^d\, .
\end{equation}
\item[(ii)]\label{item2}  For all $x\in S$, we have $\var(\varepsilon\given X=x)\leq \sigma^2$.
\item[(iii)]\label{item3}  The function $r$ is Lipschitz with constant~$1$.
\end{enumerate}
\end{Ass}

In this setting, it is known that the minimax rate of convergence for the $L_2$ risk $\E[(\hat{r}_n(X)-r(X))^2]$ is given by
\[
\left(\frac{\sigma^2}{\rho n}\right)^{\frac{2}{d+2}}\, ,
\]
up to some constants depending on $d$ only, see\cite{MR1920390}, Theorem 3.2, and~\cite{MR673642}. In particular, this rate is achieved by nearest neighbors or kernel procedures. Let us note that usually, this minimax rate is written with $D^{-d}$ instead of $\rho$, where $D$ is the diameter of $S$. However, those two quantities are clearly related by the inequality $\rho\leq D^{-d}$, which comes from~\eqref{aljcj} applied to $\varepsilon=D$. Moreover, one may check that lower bounds for the $L_2$ risk are actually obtained in situations where equation~\eqref{aljcj} is satisfied, namely with $X$ uniform on a $d$-dimensional hypercube.

The main goal of this paper is to obtain such minimax results for deviation probabilities rather than for the $L_2$ risk. More precisely, we start by establishing the following lower bound.

\begin{theorem}\label{thm:lower-bound}
For any $\delta\in ]0,1/16]$, for any $\rho>0$ and $\sigma>0$, for any regression estimate $\hat{r}_n$, there exists a distribution $Q\in\cF_{\rho,\sigma}$ as defined in Assumption \ref{aljcj} such that, when $X\sim\mu$ independent of $\cD_n$, we have
\[
\P\left(\left| \hat{r}_n(X)-r(X)\right|\geq a_d\left(\frac{\sigma^2}{\rho n}\ln\left(\frac{1}{16\delta}\right)\right)^{\frac{1}{d+2}}\right)\geq \delta\, ,
\]
where $a_d>0$ is an explicit constant depending on the dimension $d$ only.
\end{theorem}

Next, we proceed by designing estimators that achieve such deviation bounds for any distribution in $\cF_{\rho,\sigma}$. In particular, we are looking for exponential concentration with only a second moment assumption on the noise. 

Let us recall that a local averaging estimate of the regression function is an estimate that can be written as
%\begin{equation}\label{eq:local-averaging}
\[
\forall x\in\R^d\, ,\; \hat{r}_n(x):=\hat{r}_n(x,\cD_n)=\sum_{i=1}^n W_i(x)Y_i\, ,
\]
%\end{equation}
where for all $i\in\llbracket 1,n\rrbracket$, $W_i(x)$ is a Borel measurable function of $x$ and $X_1,\dots,X_n$ (but not of $Y_1,\dots,Y_n)$, with values in $[0,1]$, and such that $\sum_{i=1}^n W_i(x)=1$. This class includes nearest neighbors, kernel, and partitioning estimates.

As will be shown, given a local averaging rule, an estimator that nearly achieves the bound of Theorem~\ref{thm:lower-bound} may be constructed through the \textit{median-of-means} (MoM) technique: for $m\in\llbracket 1,n\rrbracket$, we consider $m$ disjoint subsets $\cD^{(1)},\dots,\cD^{(m)}$ of $\cD_n$, each of length $N=\left\lfloor n/m\right\rfloor$ (if $n$ is not a multiple of $m$, we simply discard some observations). For each $j\in\llbracket 1,m\rrbracket$ and all $x\in\R^d$, let 
\[
\hat{r}^{(j)}(x):=\hat{r}_N(x,\cD^{(j)})\, ,
\]
for some estimate $\hat{r}_N$, called the \textit{base estimate}. Note that, for a given $x\in\R^d$, the variables $\hat{r}^{(1)}(x),\dots,\hat{r}^{(m)}(x)$ are i.i.d., with the same distribution as $\hat{r}_N(x):=\hat{r}_N(x,\cD_N)$. The median-of-means regression estimate is then defined as
\[
\hat{r}^{\mathsf{mom}}_n(x):=\med\left(\hat{r}^{(1)}(x),\dots, \hat{r}^{(m)}(x)\right)\, ,
\]
where $\med(r_1,\dots,r_m)=r_{(\lceil m/2\rceil)}$ corresponds to the smallest value $r\in\{r_1,\dots,r_m\}$ such that
\[
\left|\left\{j\in\llbracket 1,m\rrbracket\,,\, r_j\leq r\right\}\right|\geq \frac{m}{2}\quad\mathrm{ and }\quad \left|\left\{j\in\llbracket 1,m\rrbracket\,,\, r_j\geq r\right\}\right|\geq \frac{m}{2}\, .
\]

For a variety of base estimates, including nearest neighbors and kernel estimates, we show that, when $Q\in\cF_{\rho,\sigma}$, and when $\sigma$ and $\rho$ are known, the median-of-means estimate satisfies the following deviation inequality: for all $\delta\in [e^{-nc_\cF+1},1[$, and for all $x\in S$, when the number of blocks $m$ is chosen as $m=\left\lceil \ln(
1/\delta)\right\rceil$, we have, for all $x\in S$,
\begin{equation}\label{eq:upper-bound}
\P\left( \left| \hat{r}^{\mathsf{mom}}_n(x)-r(x)\right|\geq b_d \left( \frac{\sigma^2\left\lceil \ln(1/\delta)\right\rceil}{\rho n}\right)^{\frac{1}{d+2}}\right)\leq \delta\, ,
\end{equation}
where $b_d>0$ is an explicit constant possibly depending on $d$, $c_\cF>0$ is an explicit constant that depends on $\rho,\sigma$ and $d$ only, and both $b_d$ and $c_\cF$ depend on the base estimate at stake. Specifically, this is the purpose of Propositions \ref{prop:knn} and \ref{prop:bagged} for nearest neighbors estimates, Proposition \ref{prop:kernel} for kernel estimates, and Proposition \ref{prop:partitioning} for partitioning estimates. Roughly speaking, this means that, in a large domain, the tail of $\left| \hat{r}^{\mathsf{mom}}_n(x)-r(x)\right|$ is upper-bounded by that of $Z^{\frac{2}{d+2}}$, where $Z\sim\cN(0,\frac{\sigma^2}{\rho n})$.

In fact, if for each $x\in S$, one has access to a local $\rho_x$ such that~\eqref{aljcj} is satisfied, then~\eqref{eq:upper-bound} is fulfilled with $\rho_x$ instead of $\rho$. Nevertheless, since~\eqref{eq:upper-bound} is valid for all $x\in S$, it implies that if $X\sim\mu$ independent of $\cD_n$, then
\[
\P\left( \left| \hat{r}^{\mathsf{mom}}_n(X)-r(X)\right|\geq b_d \left( \frac{\sigma^2\left\lceil \ln(1/\delta)\right\rceil}{\rho n}\right)^{\frac{1}{d+2}}\right)\leq \delta\, ,
\]
which, in view of Theorem \ref{thm:lower-bound}, is the optimal deviation bound.

\subsection{Related work}\label{sec:related-work}

To our knowledge, there are only very few results on deviation probabilities in non-parametric regression, at the exception of~\cite{jiang2019non} for the $k$-\nn\ estimate. However, in the latter, the noise is assumed sub-Gaussian. Therefore, the first contribution of the present paper is to delineate the fundamental limits of the problem by providing a general lower bound when only a second moment assumption on the noise is made (see Theorem~\ref{thm:lower-bound}). Our second contribution is to show that this bound can in fact be achieved by combining local averaging rules and the MoM principle.

Let us mention that, except for inequality \eqref{aljcj}, all points of Assumption \ref{kjncllxnlankx} (\textit{i.e.}, bounded support, bounded variance and Lipschitz property) are standard to obtain $L_2$ rates of convergence in non-parametric regression estimation, see for example Chapters 4, 5, and 6 in \cite{MR1920390} for, respectively, partitioning, kernel, and nearest neighbors estimates. Concerning \eqref{aljcj}, the proof of Theorem 1 in \cite{MR2767541}  (see in particular (13.1)) ensures that it is satisfied if $S$ is a finite union of convex bounded sets and $X$ has a density $f$ that is bounded away from zero. Such an assumption is also made by~\cite{jiang2019non}. Comparable assumptions are the so-called cone-condition in \cite{korostelev2012minimax}, Chapter 5, and the notion of standard support in \cite{MR1604449}. As will become clear in the remainder of the article, equation \eqref{aljcj} allows us to obtain inequality \eqref{eq:upper-bound} for all $x\in S$, and not only in average for $X$ with law $\mu$ (see also Remark \ref{zfoihjdoijd}). 

Concerning the MoM principle, it seems that it  was first introduced in works of~\cite{jerrum1986random},~\cite{alon1996space}, in order to obtain sub-Gaussian estimators for the mean of a heavy-tailed random variable, or when outliers may contaminate the data (see also \cite{catoni2012challenging} for a different but related approach). Some variants that do not require any knowledge on the variance have also been proposed recently, see for example \cite{lee2022optimal}, \cite{minsker2021robust}, or~\cite{gobet:hal-03631879}. One caveat of the MoM-estimator of the mean is its dependence on the confidence threshold~$\delta$. However, under stronger assumptions on the distribution, \cite{minsker2019distributed} showed that it is in fact adaptive to $\delta$ up to $\delta\approx e^{-\sqrt{n}}$. In the same vein, \cite{devroye2016sub} proposed a way to design $\delta$-independent sub-Gaussian estimators up to $\delta\approx e^{-n}$. 
 The MoM principle was also generalized to multivariate settings by~\cite{minsker2015geometric},~\cite{hsu2014heavy},~\cite{lerasle2011robust},~\cite{lugosi2019sub}, and applied to a large variety of statistical problems, including linear regression (\cite{audibert2011robust}), empirical risk minimization (\cite{brownlees2015empirical},~\cite{lecue2019learning},~\cite{lugosi2019risk}), classification (\cite{lecue2020robust}), bandits (\cite{bubeck2013bandits}), least-squares density estimation (\cite{lerasle2011robust}), and kernel density estimation (\cite{humbert2022robust}). We refer the reader to~\cite{lecue2020robustAoS},  ~\cite{lerasle2019lecture}, or~\cite{lugosi2019mean} for more references on all of these subjects.

Finally, let us note that the minimax deviation inequalities \eqref{eq:upper-bound} are obtained by optimizing on the tuning parameter of the base estimate (\textit{e.g.}, the number $k$ of neighbors for $k$-\nn\ estimates, the bandwidth $h$ for kernel estimates), and this optimization step typically requires the knowledge of $\rho$ and $\sigma$. In this respect, an open question would be to design procedures to choose this tuning parameter in an adaptive way. For the $L_2$ risk, this is possible for instance by splitting the sample or cross-validation, as explained for example in \cite{MR1920390}, Chapters 7 and 8. Unfortunately, this issue seems to be more complicated in the present case of deviation bounds.

\subsection{Organization of the paper}

In Section~\ref{sec:lower-bound}, we give the first steps for the proof of Theorem~\ref{thm:lower-bound}. Then Section~\ref{sec:preliminary} provides a guideline for proving ~\eqref{eq:upper-bound}. It also contains two general observations about our estimators: the fact that they are robust to the presence of outliers in the sample, and the fact that they can be turned into $\delta$-independent estimators, using a strategy introduced by~\cite{devroye2016sub}.  In the next two sections, Equation~\eqref{eq:upper-bound} is established for various choices of local averaging procedures, namely nearest neighbor methods in Section~\ref{sec:nearest-neighbors}, and kernel and partitioning methods in Section~\ref{sec:kernel-partitioning}. Let us point out that, for partitioning estimates (Section~\ref{subsec:partitioning}), we are able to obtain a uniform control on $\sup_{x\in S}|\hat{r}_n^\mathsf{mom}(x)-r(x)|$. Finally, Section~\ref{sec:proofs} gathers the proofs of several technical results.

\section{Lower bound}\label{sec:lower-bound}
In order to establish Theorem~\ref{thm:lower-bound}, we consider the following setting. For some $A>0$, let 
\[
S=[0\, ,\, A]^d\, ,
\]
and consider the model
\[
Y=r(X)+\varepsilon\, ,
\]
where $X\sim\mu=\mathrm{Unif}(S)$, where the noise $\varepsilon\sim\cN(0,\sigma^2)$ is independent of $X$, and where $r$ is a $1$-Lipschitz function on $S$. Notice that this model satisfies Assumption \ref{kjncllxnlankx}, meaning that $Q\in\mathcal{F}$. Indeed, the support $S$ is bounded with diameter $D=A\sqrt{d}$ and, for all $x\in S$ and $\varepsilon\in(0,D]$, one has 
\[
\mu\left(\cB(x,\varepsilon)\right)\geq \mu\left(\cB(x,\varepsilon/\sqrt{d})\right)\geq\mu\left(\cB(0,\varepsilon/\sqrt{d})\right)=\frac{\pi^{d/2}}{(2A)^d d^{d/2}\Gamma(1+d/2)}\varepsilon^d=:\rho\varepsilon^d\, ,
\]
and any value of $\rho>0$ can be achieved by adjusting the sidelength~$A$.

%In this setting, we show that for any $\delta\in ]0,1/16]$, for any regression estimate $\hat{r}_n$, there exists a 1-Lipschitz mapping $r=r_n:S\to\R$ such that, when $X\sim\mu$ independent of $\cD_n$, we have
%\begin{equation}\label{eq:lower-boundbis}
%\P\left(\left| \hat{r}_n(X)-r(X)\right|\geq \frac{1-2^{-\frac{1}{d}}}{2}\left(\frac{\pi(d+1)^2\sigma^2A^d}{n}\ln\left(\frac{1}{16\delta}\right)\right)^{\frac{1}{d+2}}\right)\geq \delta\, .
%\end{equation}
%

The starting point of the proof uses the same idea as in the proof of Theorem 3.2 in \cite{MR1920390} (see also Section 2 in \cite{MR673642}). Namely, let $\mathcal C:=\left[-\tfrac{1}{2},\tfrac{1}{2}\right]^d$ and $\partial\mathcal C$ its frontier, then define the $1$-Lipschitz function
\[
g:x\mapsto \mathrm{dist}(x,\partial\mathcal C)\ \mathbf{1}_{x\in\mathcal C}=\inf\{\|x-y\|,y\in\partial\mathcal C\}\ \mathbf{1}_{x\in\mathcal C}\, .
\]
Let $M\geq 1$ be some integer to be specified later and consider a partition of $S$ by $K:=M^d$ hypercubes $A_j$ of sidelength $h:=A/M$ and with centers $a_j$, and let the functions $g_1,\dots,g_{K}$ be defined by
\[
\forall j\in\llbracket 1,K\rrbracket\, ,\; g_j(x):=h g\left(h^{-1}\left(x-a_j\right)\right)\, .
\]
%Hence the support of $g_j$ is $A_j=\left[a_j-\tfrac{h}{2};a_j+\tfrac{h}{2}\right]^d$. 
We then look for a ``bad'' regression function among the functions
\[
r^{(c)}=\sum_{j=1}^{K} c_j g_j\, ,\, c\in\{-1,1\}^{K}\, ,
\]
which are all $1$-Lipschitz by~\cite{MR1920390}. In this setting, we are led to the following result, which is proved in Section~\ref{subsec:proof-lower}. 

\begin{proposition}\label{prop:lower-bis}
For any $\delta\in ]0,1/16]$, for any regression estimate $\hat{r}_n$, there exists  $c\in\{-1,1\}^K$ such that, when $X\sim\mu$ independent of $\cD_n$, we have
\begin{equation}
\P\left(\left| \hat{r}_n(X)-r^{(c)}(X)\right|\geq \frac{1-2^{-\frac{1}{d}}}{2}\left(\frac{\pi(d+1)^2\sigma^2A^d}{n}\ln\left(\frac{1}{16\delta}\right)\right)^{\frac{1}{d+2}}\right)\geq \delta\, .
\end{equation}
\end{proposition}

Since $A^d=\frac{\pi^{d/2}}{2^d d^{d/2}\Gamma(1+d/2)} \rho^{-1}$, this establishes Theorem~\ref{thm:lower-bound} with
\[
a_d=  \frac{1-2^{-\frac{1}{d}}}{2} \left(\frac{\pi^{1+d/2}(d+1)^2}{2^d d^{d/2}\Gamma(1+d/2)}\right)^{\frac{1}{d+2}}\, .
\]

%\begin{remark}\label{zfoihjlsjcndoijd} The proof of Proposition~\ref{prop:lower-bis} uses a Bayesian argument, also known as the probabilistic method: we show that it holds on average for a uniform random vector $C=(C_1,\dots,C_{K})\in\{-1,1\}^{K}$, that is, i.i.d.\ Rademacher random variables with parameter $1/2$.
%\end{remark}

\section{Median-of-means versions of local averaging procedures}\label{sec:preliminary}

\subsection{Two key lemmas}

This section exposes two generic results that will be of constant use for proving upper bounds. The first lemma relates deviation probabilities for the median-of-means estimate $\hat{r}^{\mathsf{mom}}_n$ with deviation probabilities of the base estimate $\hat{r}_N$. We point out that this result is valid for any base estimate $\hat{r}_N$, not only for local averaging rules.

\begin{lemma}\label{lem:binomial}
Let $\hat{r}^{\mathsf{mom}}_n$ be the median-of-means estimate of $r$ constructed on $m$ blocks with base estimate $\hat{r}_N$. For all $x\in \R^d$ and all $t\geq 0$, we have
\[
\P\left( \left| \hat{r}^{\mathsf{mom}}_n(x)-r(x)\right|\geq t\right)\leq 2^m p_t(x)^{m/2} \, ,
\]
where 
\[
p_t(x):=\P\left(\left| \hat{r}_N(x)-r(x)\right|\geq t \right)\, .
\]
\end{lemma}

Assume now that the base estimate takes the form 
\begin{equation}\label{eq:local-averaging}
\forall x\in\R^d\, ,\; \hat{r}_N(x)=\hat{r}_N(x,\cD_N)=\sum_{i=1}^NW_i(x)Y_i\, ,
\end{equation}
where for all $i\in\llbracket 1,N\rrbracket$, $W_i(x)$ is a Borel measurable function of $x$ and $X_1,\dots,X_N$ (but not of $Y_1,\dots,Y_N)$, with values in $[0,1]$, and such that $\sum_{i=1}^N W_i(x)=1$. Our second lemma gives a bias--variance decomposition for deviation probabilities of local averaging estimates.

\begin{lemma}\label{lem:bias-variance}
Suppose that $(ii)$ and $(iii)$ in Assumption \ref{kjncllxnlankx} are satisfied. Then, for all $x\in \R^d$, we have
\begin{equation}\label{eq:bias-variance1}
\left|\hat{r}_N(x)-r(x)\right|\leq \left|\sum_{i=1}^N W_i(x)\varepsilon_i\right|+\sum_{i=1}^N W_i(x)\|X_i-x\|\, ,
\end{equation}
where $\varepsilon_i=Y_i-r(X_i)$. In addition, for all $s,t>0$, we have
\begin{equation}\label{eq:bias-variance2}
p_{t+s}(x)\leq \frac{\sigma^2 \E\left[\sum_{i=1}^N W_i(x)^2\right]}{t^2}+ \P\left( \sum_{i=1}^N W_i(x)\|X_i-x\|\geq s\right)\, ,
\end{equation}
with $p_{t+s}(x)$ as defined in Lemma \ref{lem:binomial}.
\end{lemma}

In Sections \ref{sec:nearest-neighbors} and \ref{sec:kernel-partitioning}, we will investigate several instances of local averaging procedures for the base estimate $\hat{r}_N$. In each case, we first use Lemma~\ref{lem:bias-variance} in order to determine $t$ and $s$ such that $p_{t+s}(x)\leq \frac{1}{4e^2}$. Lemma~\ref{lem:binomial} then entails that
\[
\P\left( \left| \hat{r}^{\mathsf{mom}}_n(x)-r(x)\right|\geq t+s\right)\leq e^{-m}\, .
\]
The number of blocks $m$ can then be chosen as $\lceil \ln(1/\delta)\rceil$, for some target confidence threshold $\delta\in [e^{-n},1[$, so that the probability above is less than $\delta$. Next, provided $\sigma$ and $\rho$ are known, a tuning parameter of the base estimate (\textit{e.g.}, the number $k$ of neighbors for $k$-\nn\ estimates, the bandwidth $h$ for kernel estimates) can then be optimized to get a bound in the flavor of~\eqref{eq:upper-bound}, imposing additional constraints on $\delta$.

\subsection{Robustness}\label{subsec:robustness}

In addition to exhibiting concentration properties, the estimate $\hat{r}^{\mathsf{mom}}_n$ also stands out through its strong robustness to outliers. In this section, we consider a general contamination scheme which was introduced by~\cite{diakonikolas2019robust}, and which was also considered in~\cite{lecue2019learning}. We assume that the index set $\{1,\dots,n\}$ is divided into two disjoint subsets: the subset $\cI$ of inliers, and the subset $\cO$ of outliers. The sequence $(X_i,Y_i)_{i\in\cI}$ is i.i.d.\ with the same law as $(X,Y)\sim Q\in\cF$. No assumption is made on the variables $(X_i,Y_i)_{i\in\cO}$. We denote by $\P_{\cO\cup\cI}$ the distribution corresponding to such a contaminated sample.

Let $\hat{r}_n^{\mathsf{mom}}$ be the median-of-means estimate of $r$ constructed on $m$ blocks, with base estimate $\hat{r}_N$. Assume that the number $|\cO|$ of outliers in the original sample is stricly less than $m/2$. For simplicity, we here further assume that $|\cO|\leq m/4$. Then, letting $\cB$ be the set of blocks that do not intersect $\cO$, we have, for all $s,t> 0$,
\begin{align*}
\P_{\cO\cup\cI}\left( \left|\hat{r}_n^{\mathsf{mom}}(x)-r(x)\right|\geq t+s\right)&\leq \P_{\cO\cup\cI}\left(\sum_{j=1}^m \ind_{\left\{ \left|\hat{r}^{(j)}(x)-r(x)\right|\geq t+s\right\}} \geq \frac{m}{2}\right)\\
&\leq \P_{\cO\cup\cI}\left(|\cB^c|+\sum_{j\in\cB} \ind_{\left\{ \left|\hat{r}^{(j)}(x)-r(x)\right|\geq t+s\right\}} \geq \frac{m}{2}\right)\, .
\end{align*}
Observing that $|\cB^c|\leq |\cO|\leq m/4$, we get, with a slight abuse of notation,
\[
\P_{\cO\cup\cI}\left( \left|\hat{r}_n^{\mathsf{mom}}(x)-r(x)\right|\geq t+s\right)\leq  \P\left(\sum_{j=1}^m \ind_{\left\{ \left|\hat{r}^{(j)}(x)-r(x)\right|\geq t+s\right\}} \geq \frac{m}{4}\right)\, .
\]
Now, the proof of Lemma~\ref{lem:binomial} reveals that
\[
\P\left(\sum_{j=1}^m \ind_{\left\{ \left|\hat{r}^{(j)}(x)-r(x)\right|\geq t+s\right\}} \geq \frac{m}{4}\right)\leq \left(\frac{4}{3^{3/4}}\right)^mp_{t+s}(x)^{m/4}\, ,
\]
which is less than $e^{-m}$ for $p_{t+s}(x)\leq \frac{27}{(4e)^4}$. Hence, provided $m\geq 4|\cO|$ and modulo some minor changes in the constants, the same strategy as the one described just after the statement of Lemma~\ref{lem:bias-variance} can be applied to the contaminated setting. 

\begin{remark}\label{zhcb}
More generally, define $\alpha$ as the proportion of outliers in the original sample and assume that $m >2|\cO|=2\alpha n$, then a straightforward adaptation of the previous reasoning shows that 
\[
\P\left(\sum_{j=1}^m \ind_{\left\{ \left|\hat{r}^{(j)}(x)-r(x)\right|\geq t+s\right\}} \geq \frac{m}{2}-\alpha n\right)\leq 2^mp_{t+s}(x)^{\frac{m}{2}-\alpha n}\, ,
\]
which is less than $e^{-m}$ for $p_{t+s}(x)\leq (2e)^{-\frac{2m}{m-2\alpha n}}$. Results in the same spirit are available in the literature for the problems of mean estimation~\cite{minsker2019uniformboundsrobustmean}, Lemma 5.4, and robust regression ~\cite{oliveira2023trimmedsamplemeansrobust}, Lemma D.1.

\end{remark}

%For example, in the case of uniform nearest neighbors detailed in Section \ref{zdpzdjp}, some computation shows that the estimator $\hat{r}^{\mathsf{mom}}_n$ constructed on $m$ blocks with $k^\star$-\nn\ base estimators, where
%\[
%k^\star=\left\lfloor \left(\frac{1}{8}(2e)^{-\frac{2m}{m-2\varepsilon n}}\sigma^2\left(\frac{\rho n}{m}\right)^{\frac{2}{d}}\right)^{\frac{d}{d+2}}\right\rfloor\, ,
%\]
%which belongs to $\llbracket 1,N\rrbracket=\llbracket 1,\left\lfloor \frac{n}{m}\right\rfloor\rrbracket$ provided
%\[
%1\leq\left(\frac{1}{8}(2e)^{-\frac{2m}{m-2\varepsilon n}}\sigma^2\left(\frac{\rho n}{m}\right)^{\frac{2}{d}}\right)^{\frac{d}{d+2}} \leq \frac{n}{m}\, ,
%\]
%satisfies
%\[
%\P\left( \left| \hat{r}^{\mathsf{mom}}_n(x)-r(x)\right|\geq 8\sqrt{2}(2e)^{\frac{2m}{m-2\varepsilon n}}\left(\frac{\sigma^2\left\lceil \ln(1/\delta)\right\rceil}{\rho n}\right)^{\frac{1}{d+2}}\right)\leq \delta\, .
%\]

\subsection{From $\delta$-dependent to $\delta$-independent estimators}

One feature  of $\hat r^{\mathsf{mom}}_n(x)$ is that it depends on $m$, the number of blocks, and thus on the pre-chosen  confidence threshold $\delta=e^{-m}$. In this section, we give an argument due to~\cite{devroye2016sub}, allowing to turn $\hat r^{\mathsf{mom}}_n(x)$ into an estimator satisfying~\eqref{eq:upper-bound} simultaneously for an infinity of $\delta$. Let us first recall that, when $\sigma$ and $\rho$ are known, then for all integer $m$ between $1$ and $\lfloor c n\rfloor$ (for some constant $c$ depending only on $\rho$, $\sigma$ and $d$), one is able to construct a confidence interval $\hat{I}_m$ for $r(x)$ with level $1-e^{-m}$. Now let
\[
\hat{m}:=\min\left\{ 1\leq m\leq \lfloor c n\rfloor\, , \bigcap_{j=m}^{\lfloor c n\rfloor} \hat{I}_j\neq \emptyset\right\}\, ,
\]
and define the estimator $\tilde{r}_n(x)$ as the midpoint of the interval $\bigcap_{j=\hat{m}}^{\lfloor c n\rfloor} \hat{I}_j$. Let $\delta\in \left[ \frac{e^{-\lfloor c n\rfloor}}{1-e^{-1}},1\right[$  and let $m_\delta$ be the smallest integer $m\in \llbracket 1,\lfloor c n\rfloor\rrbracket$ such that $\delta\geq \frac{e^{-m}}{1-e^{-1}}$, namely $m_\delta=\left\lceil \ln\left(\frac{1}{(1-e^{-1})\delta}\right)\right\rceil$. Then $\tilde{r}_n(x)$ satisfies
\begin{equation}\label{eq:median-of-means-independent}
\P\left( \left| \tilde{r}_n(x) -r(x)\right|>  | \hat{I}_{m_\delta}|\right)\leq \delta\, .
\end{equation}
 Indeed, by a union bound, we have
\begin{align*}
\P\left( \bigcap_{j=m_\delta}^{\lfloor c n\rfloor} \left\{ r(x) \in \hat{I}_j\right\}\right) & \geq 1-\sum_{j=m_\delta}^{\lfloor c n\rfloor} e^{-j} \geq 1-\frac{e^{-m_\delta}}{1-e^{-1}} \geq 1-\delta\, .
\end{align*}
Now, on the event $\bigcap_{j=m_\delta}^{\lfloor c n\rfloor} \left\{ r(x) \in \hat{I}_j\right\}$, one has $\bigcap_{j=m_\delta}^{\lfloor c n\rfloor} \hat{I}_j\neq \emptyset$, hence $\hat{m}\leq m_\delta$. But if $\hat{m}\leq m_\delta$, then $\tilde{r}_n(x)$ also belongs to $\bigcap_{j=m_\delta}^{\lfloor c n\rfloor} \hat{I}_j$, and in particular
\[
\left| \tilde{r}_n(x) -r(x)\right|\leq | \hat{I}_{m_\delta}|\, ,
\]
which establishes~\eqref{eq:median-of-means-independent}.

%For instance, for $k$-\nn\ base estimates, one may combine Proposition~\ref{prop:knn} and this method to construct an estimator $\tilde{r}_n(x)$ which is such that, for all $\delta\in\left[\frac{e^{-cn+1}}{1-e^{-1}}, 1\right[$,
%\[
%\P\left( \left| \tilde{r}_n(x)-r(x)\right|\geq 64e^2\sqrt{2}\left(\frac{\sigma^2\left\lceil \ln\left(\frac{1}{(1-e^{-1})\delta}\right)\right\rceil}{\rho n}\right)^{\frac{1}{d+2}}\right)\leq \delta\, ,
%\]
%with 
%\[
%c=\rho\left(\frac{\sigma}{4e\sqrt{2}}\right)^d\wedge \frac{32e^2}{\sigma^2\rho^{2/d}}\wedge 1\, .
%\]

\section{Nearest neighbors estimation}\label{sec:nearest-neighbors}

For $x\in\R^d$, let 
\[
\left(X_{(1)}(x),Y_{(1)}(x)\right),\dots, \left(X_{(N)}(x),Y_{(N)}(x)\right)
\]
a reordering of the data $\cD_N$ according to increasing values of $\|X_i-x\|$, that is
\[
\|X_{(1)}(x)-x\|\leq \dots\leq \|X_{(N)}(x)-x\|\, ,
\]
where, if necessary, distance ties are broken by simulating auxiliary random variables $(U_1,\dots,U_N)$ i.i.d.\ with uniform law on $[0,1]$ and sorting them. The weighted nearest neighbors estimate is defined as 
\begin{equation}\label{alkcjc}
\hat{r}_N(x):=\sum_{i=1}^N v_i Y_{(i)}(x)\, ,
\end{equation}
where $(v_1,\dots,v_N)$ is a deterministic vector in $[0,1]^N$ satisfying $\sum_{i=1}^N v_i=1$. Note that this estimate is of the form~\eqref{eq:local-averaging}, with $W_i(x)=v_{\sigma(i)},$ where $\sigma$ is a random permutation (depending on $x$) such that $X_i=X_{(\sigma(i))}(x)$. We refer the interested reader to \cite{MR3445317}, Chapter 8, or  \cite{samworth2012optimal} and references therein for more information on this topic.

In this context, the variance term of Lemma \ref{lem:bias-variance} reduces to
\begin{equation}\label{eq:variance-nearest-neighbors}
\E\left[\sum_{i=1}^N W_i(x)^2\right]=\sum_{i=1}^N v_{\sigma(i)}^2 =\sum_{i=1}^N v_i^2\, \cdot 
\end{equation}
As for the bias term, letting $D_{(i)}(x):=\|X_{(i)}(x)-x\|$, we rewrite it as 
\begin{equation}\label{peijcfpeijfc}
\sum_{i=1}^N W_i(x)\|X_i-x\|=\sum_{i=1}^Nv_i D_{(i)}(x)\, .
\end{equation}
The following lemma, whose proof is housed in Section~\ref{sec:proofs}, allows us to control the expected nearest neighbor distances (see Remark \ref{zfoihjdoijd} below for a comment on this result). 

\begin{lemma}\label{lem:expected-distance}
Under Assumption \ref{kjncllxnlankx}$(i)$, for all $x\in S$ and $i\in\llbracket 1,N\rrbracket$, one has
\[
\E\left[ D_{(i)}(x)\right]\leq 2\left(\frac{i}{\rho(N+1)}\right)^{1/d}\, .
\]
\end{lemma}

According to \eqref{peijcfpeijfc}, we then deduce from Markov's inequality that, for all $x\in S$,
\begin{equation}\label{eq:bias-nearest-neighbors}
\P\left( \sum_{i=1}^N W_i(x)\|X_i-x\|\geq s\right)\leq \frac{2}{s}\sum_{i=1}^N v_i\left(\frac{i}{\rho(N+1)}\right)^{1/d}\, .
\end{equation}
Combining~\eqref{eq:variance-nearest-neighbors} and~\eqref{eq:bias-nearest-neighbors}, and applying Lemma~\ref{lem:bias-variance} with
\begin{equation}\label{eicjpaec}
t=2e\sigma\sqrt{2\sum_{i=1}^Nv_i^2}
\qquad \text{ and }\qquad
s=16e^2\sum_{i=1}^Nv_i\left(\frac{i}{\rho(N+1)}\right)^{1/d}\, ,
\end{equation}
we see that, for all $x\in S$,
\[
p_{t+s}(x)\leq \frac{1}{4e^2}\, ,
\]
which entails, by Lemma~\ref{lem:binomial}, that
\begin{equation}\label{eicjpaecqkjx}
\P\left( \left| \hat{r}^{\mathsf{mom}}_n(x)-r(x)\right|\geq t+s\right)\leq e^{-m}\, .
\end{equation}

We first propose to illustrate this result on two specific examples of nearest neighbors rules:  the uniform $k$-nearest neighbors estimate (Subsection \ref{subsec:knn}) and the bagged 1-nearest neighbor estimate (Section \ref{subsec:bagging}). As we will see, both satisfy the deviation  inequality \eqref{eq:upper-bound}. 

\begin{remark}\label{zfoihjdoijd} The fact that the upper bound of Lemma \ref{lem:expected-distance} is valid for all $d\geq 1$ and, more importantly, for all $x\in S$, is due to inequality \eqref{aljcj} in Assumption \ref{kjncllxnlankx}. If one only supposes that the support of $X$ is bounded, then \cite{MR2600626}, Corollary 2.1, for $d=1$ or $d\geq 3$ and a consequence of \cite{MR2584901}, Theorem 3.2, when $d=2$, only ensure that there exists a constant $c_d$ depending on the dimension $d$ and the size of the support such that, for all $d\geq 2$,
\[
\E\left[ D_{(i)}(X)^2\right]\leq c_{d}\left\lfloor\frac{N}{i}\right\rfloor^{-2/d}\, ,
\]
whereas for $d=1$ one has $\E\left[ D_{(i)}(X)^2\right]\leq c_{1}\frac{i}{N}$.
\end{remark}

\subsection{The $k$ Nearest Neighbors estimate}
\label{subsec:knn}
Let us focus on the case of uniform $k$-nearest neighbors ($k$-\nn). Namely, we now set
\[
\hat{r}_N(x):=\frac{1}{k}\sum_{i=1}^k Y_{(i)}(x)\, ,
\]
for some $k\in\llbracket 1,N\rrbracket$. 
%Then, for all $x\in S$,
%\[
%v(x)\leq \frac{1}{k}\qquad \text{ and }\qquad b(x)\leq \frac{2}{c^{1/d}}\left(\frac{k}{N+1}\right)^{1/d}\, .
%\]
We then have the following proposition, whose proof can be found in Section~\ref{subsec:proof-knn}.

\begin{proposition}\label{prop:knn}
Under Assumption \ref{kjncllxnlankx}, let
\[
c=\rho\left(\frac{\sigma}{4e\sqrt{2}}\right)^d\wedge \frac{32e^2}{\sigma^2\rho^{2/d}} \; , \quad \delta\in[e^{-\lfloor cn\rfloor}, 1[\;, \quad\text{ and }\quad m=\left\lceil \ln(1/\delta)\right\rceil\; ,
\]
ensuring that $1\leq m\leq cn$. Then the estimator $\hat{r}^{\mathsf{mom}}_n$ constructed on $m$ blocks with $k^\star$-\nn\ base estimators, where
\[
k^\star=\left\lfloor \left(\frac{\sigma^2}{32e^2}\right)^{\frac{d}{d+2}}\left(\frac{\rho n}{m}\right)^{\frac{2}{d+2}}\right\rfloor\, ,
\]
satisfies
\[
\P\left( \left| \hat{r}^{\mathsf{mom}}_n(x)-r(x)\right|\geq 32e^2\sqrt{2}\left(\frac{\sigma^2\left\lceil \ln(1/\delta)\right\rceil}{\rho n}\right)^{\frac{1}{d+2}}\right)\leq \delta\, .
\]
\end{proposition}

Hence, when $k^\star$-\nn\ is chosen as base estimate, inequality~\eqref{eq:upper-bound} is satisfied with $b_d=32e^2\sqrt{2}$ and $c_\cF=c$.

\begin{remark} One may notice that the optimal value for $k$ has the same dependency with respect to $\sigma^2$ and $n$, that is $k^\star=O(\sigma^{\frac{2d}{d+2}}n^{\frac{2}{d+2}})$, as the one that balances bias and variance when minimizing the $L_2$ risk, see \cite{MR1920390} Theorem 6.2. In a different setting, the conclusion is the same in the work of Jiang, see Remark 1 in \cite{jiang2019non}. 
\end{remark}

\subsection{Bagging and Nearest Neighbors}
\label{subsec:bagging}

We now turn to the bagged $1$-\nn\ estimate with replacement. Bagging (for \textbf{b}ootstrap \textbf{agg}regat\textbf{ing}) is a simple way to combine estimates in order to improve their performance. This method, suggested by \cite{breiman1996bagging}, proceeds
by resampling from the original data set, constructing a predictor from
each subsample, and decide by combining. By bagging an $N$-sample,
the crude nearest neighbor regression estimate is turned into a consistent weighted nearest neighbor regression estimate, which is amenable
to statistical analysis. In particular, one may find experimental performances, consistency results, rates of convergence, and minimax properties in \cite{hall2005properties}, \cite{biau2010layered}, \cite{MR2600626}, to cite just a few references.

Without going into details, it turns out that bagged $1$-\nn\ estimates, with or without replacement, can easily be reformulated as weighted nearest neighbor rules (see for example \cite{biau2010layered}). For sampling without replacement, the weights in \eqref{alkcjc} are, for some $k\in\llbracket 1,N\rrbracket$,
\[
v_i:=\frac{{N-i \choose k-i}}{{N \choose k}}\mathbf{1}_{i\in\llbracket 1,N-k+1\rrbracket}\, ,
\]
while for sampling with replacement, we get
\[
v_i:=\left(1-\frac{i-1}{N}\right)^k-\left(1-\frac{i}{N}\right)^k\, .
\]
Here we focus on sampling with replacement and get the following bound, whose proof can be found in Section~\ref{subsec:proof-knn}. The case of sampling without replacement leads to the same type of result.

\begin{proposition}\label{prop:bagged}
Under Assumption \ref{kjncllxnlankx}, let
\[
c=\rho\left(\frac{\sigma}{2e\sqrt{2}}\right)^d \wedge \frac{8e^2}{\sigma^2\rho^{2/d}}\; ,\quad \delta\in[e^{-\lfloor cn\rfloor }, 1[\; , \quad\text{ and }\quad m=\left\lceil \ln(1/\delta)\right\rceil\; ,
\]
ensuring that $1\leq m\leq cn$. Then the estimator $\hat{r}^{\mathsf{mom}}_n$ constructed on $m$ blocks with $k^\star$-bagged $1$-\nn\ base estimators, where
\[
k^\star =\left\lfloor \left(\frac{8e^2n}{\rho^{2/d}\sigma^2m}\right)^{\frac{d}{d+2}}\right\rfloor\, ,
\]
satisfies
\[
\P\left( \left| \hat{r}^{\mathsf{mom}}_n(x)-r(x)\right|\geq 64 e^3\left(\frac{\sigma^2\left\lceil \ln(1/\delta)\right\rceil}{\rho n}\right)^{\frac{1}{d+2}}\right)\leq \delta\, .
\]
\end{proposition}

Hence, for $k^\star$-bagged $1$-\nn\ base estimate, inequality \eqref{eq:upper-bound} still holds, with this time $b_d=64e^3$ and $c_\cF=c$.

\section{Kernel and partitioning estimation}\label{sec:kernel-partitioning}

\subsection{Kernel estimates}

Let $0<h\leq D$ and consider the kernel estimator
\[
\hat{r}_N(x):=\frac{1}{N_h(x)}\sum_{i=1}^N Y_i\ind_{\|X_i-x\|\leq h}\, ,
\]
where
\[
N_h(x):=\sum_{i=1}^N \ind_{\|X_i-x\|\leq h}\, ,
\]
with the convention $0/0=0$. Observe that this is again of the form~\eqref{eq:local-averaging} with 
\[W_i(x)=N_h(x)^{-1}\ind_{\|X_i-x\|\leq h}\, .
\]

\begin{proposition}\label{prop:kernel}
Under Assumption \ref{kjncllxnlankx}, let
\[
c=\frac{\rho D^{d+2}}{8e^2\sigma^2}\wedge 1 \; ,\quad  \delta\in[e^{-\lfloor cn\rfloor}, 1[\; ,\quad\text{ and }\quad m=\left\lceil \ln(1/\delta)\right\rceil\; ,
\]
ensuring that $1\leq m\leq cn$. Then the estimator $\hat{r}^{\mathsf{mom}}_n$ constructed on $m$ blocks with $h^\star$-kernel base estimators, where
\[
h^\star=\left(\frac{8e^2\sigma^2m}{\rho n}\right)^{\frac{1}{d+2}}\, ,
\]
satisfies
\[
\P\left( \left| \hat{r}^{\mathsf{mom}}_n(x)-r(x)\right|\geq 4e^{2/3} \left(\frac{\sigma^2\left\lceil \ln(1/\delta)\right\rceil}{\rho n}\right)^{\frac{1}{d+2}}\right)\leq \delta\, .
\]
\end{proposition}

In other words, inequality \eqref{eq:upper-bound} is fulfilled with $b_d=4e^{2/3}$ and $c_\cF=c$. The proof can be found in Section~\ref{subsec:proof-kernel}.

\subsection{Partitioning estimates}
\label{subsec:partitioning}

To simplify the presentation, let us here assume that $S=[0,1]^d$. For some integer $K\geq 1$, let $\mathcal{P}=\{A_{1},A_{2},\dots,A_{K^d}\}$ be a cubic partition of $[0,1]^d$ by $K^d$ cubes with side length $1/K$. For $k\in\llbracket 1,K^d\rrbracket$, if $x\in A_k$, the partitioning estimate of the regression function takes the form
\[
\hat{r}_N(x)=\sum_{i=1}^N W_i(x)Y_i:=\frac{1}{N_k}\sum_{i=1}^N Y_i\ind_{X_i\in A_k}\, ,
\]
where
\[
N_k:=\sum_{i=1}^N \ind_{X_i\in A_k}\, ,
\]
with the usual convention $0/0=0$. 

The upcoming result shows that inequality \eqref{eq:upper-bound} is then fulfilled with $b_d=16e\sqrt{d}$ and $c_\cF=c$.

\begin{proposition}\label{prop:partitioning}
Under Assumption \ref{kjncllxnlankx} with $S=[0,1]^d$, let
\[
c=\frac{\rho d }{2^{d+3}e^2\sigma^2}\wedge 1\; ,\quad \delta\in[e^{-cn+1}, 1[\; ,\quad\text{ and }\quad m=\left\lceil \ln(1/\delta)\right\rceil\; ,
\]
ensuring that $1\leq m\leq cn$. Then the estimator $\hat{r}^{\mathsf{mom}}_n$ on $[0,1]^d$ constructed on $m$ blocks with partitioning base estimators on $K_\star^d$ hypercubes, where
\[
K_\star=\left\lfloor \left(\frac{\rho d n}{2^{d+3}e^2\sigma^2m}\right)^{\frac{1}{d+2}}\right\rfloor\, ,
\]
satisfies
\[
\P\left( \left| \hat{r}^{\mathsf{mom}}_n(x)-r(x)\right|\geq 16e\sqrt{d}\left(\frac{\sigma^2\left\lceil \ln(1/\delta)\right\rceil}{\rho n}\right)^{\frac{1}{d+2}}\right)\leq \delta\, .
\]
\end{proposition}

One enjoyable feature of partitioning estimates is that uniform bounds for $x\in S$ can easily be obtained.

\begin{proposition}\label{prop:partitioningbis}
Under Assumption \ref{kjncllxnlankx} with $S=[0,1]^d$, let
\[
c=\frac{\rho d }{2^{d+3}e^2\sigma^2}\wedge 1\; ,\quad \delta\in[e^{-cn+1}, 1[\; ,\quad\text{ and }\quad m=\left\lceil \ln(1/\delta)\right\rceil\; ,
\]
ensuring that $1\leq m\leq cn$. Then the estimator $\hat{r}^{\mathsf{mom}}_n$ on $[0,1]^d$ constructed on $m$ blocks with partitioning base estimators on $K_\star^d$ hypercubes, where
\[
K_\star=\left\lfloor \left(\frac{\rho d n}{2^{d+3}e^2\sigma^2m}\right)^{\frac{1}{d+2+\frac{2d}{m}}}\right\rfloor\, ,
\]
satisfies
\[
\P\left( \sup_{x\in S}\left| \hat{r}^{\mathsf{mom}}_n(x)-r(x)\right|\geq 16e\sqrt{d}\left(\frac{\sigma^2\left\lceil \ln(1/\delta)\right\rceil}{\rho n}\right)^{\frac{1}{d+2+\frac{2d}{\left\lceil \ln(1/\delta)\right\rceil}}}\right)\leq \delta\, .
\]
\end{proposition}
In particular, we see that if $m=m_n$ satisfies $m=o(n)$ and $e^{-m}$ is summable (\textit{e.g.}, $m=(\log n)^2$), then Borel--Cantelli Lemma entails that
\[
\sup_{x\in S}\left|\hat{r}^{\mathsf{mom}}_n(x)-r(x)\right|\xrightarrow[n\to\infty]{}0\quad\text{ almost surely.}
\]

%In particular, we see that if $K\to +\infty$, and if there exists $\varepsilon>0$ such that $n^{-1}K^{d+\varepsilon}\to 0$, then, choosing $m$ such that $n^{-1}mK^{d+\frac{2d}{m}}\to 0$ and $e^{-m}$ is summable (\textit{e.g.}, $m=(\log n)^2$), Borel--Cantelli Lemma entails that
%\[
%\sup_{x\in S}\left|\hat{r}^{\mathsf{mom}}_n(x)-r(x)\right|\,\underset{n\to +\infty}\longrightarrow\, 0\quad\text{ almost surely.}
%\]

\section{Proofs}
\label{sec:proofs}

\subsection{Proofs from Section~\ref{sec:lower-bound}}
\label{subsec:proof-lower}

\begin{proof}[Proof of Proposition~\ref{prop:lower-bis}]
Recall that for $\mathcal C:=\left[-\tfrac{1}{2},\tfrac{1}{2}\right]^d$ and $\partial\mathcal C$ its frontier, the function $g$ is defined by
\[
g(x)=\mathrm{dist}(x,\partial\mathcal C)\ \mathbf{1}_{x\in\mathcal C}=\inf\{\|x-y\|,y\in\partial\mathcal C\}\ \mathbf{1}_{x\in\mathcal C}\, .
\]
This function is $1$-Lipschitz and one can check that
\[
\int g(x)^2 \d x=\frac{1}{2(d+1)(d+2)}\, .
\]
Next, for an integer $M\geq 1$ to be specified later, consider a partition of $S$ by $K:=M^d$ hypercubes $A_j$ of sidelength $h:=A/M$ and with centers $a_j$, and let the functions $g_1,\dots,g_{K}$ be defined by
\[
\forall j\in\llbracket 1,K\rrbracket\, ,\; g_j(x):=h g\left(h^{-1}\left(x-a_j\right)\right)\, .
\]
Hence the support of $g_j$ is $A_j=\left[a_j-\tfrac{h}{2};a_j+\tfrac{h}{2}\right]^d$ and
 \begin{equation}\label{aozchaeicjoazi}
\int g_j(x)^2 \d x=\frac{h^{d+2}}{2(d+1)(d+2)}\, .
\end{equation} 
We then restrict to regression functions of the form
\[
r^{(c)}=\sum_{j=1}^{K} c_j g_j\, ,\, c\in\{-1,1\}^{K}\, .
\]

Now let $\hat{r}_n$ denote any regression estimate. For $j\in\llbracket 1,K\rrbracket$, and for $x\in A_j$, we start by defining
\[
\tilde{r}_n(x):=\mathrm{sign}\left(\hat{r}_n(x)\right)g_j(x)\, .
\]
Note that for all $x\in S$, and $c\in\{-1,1\}^{K}$, we have
\begin{equation}\label{aozchaeicjoaoaijcoazijzi}
\left| \hat{r}_n(x)-r^{(c)}(x)\right| \geq \frac{1}{2} \left| \tilde{r}_n(x)-r^{(c)}(x)\right| \, .
\end{equation}
We proceed by designing estimated signs $\tilde{c}_j$ as follows: for all  $j\in\llbracket 1,K\rrbracket$, define the hypercube 
\[
A'_j:=\left[a_j-\tfrac{h}{2^{1+\frac{1}{d}}};a_j+\tfrac{h}{2^{1+\frac{1}{d}}}\right]^d\subset A_j
\] 
and set
\[
\tilde{c}_j:=
\begin{cases}
+1 &\text{ if $\lambda\left(\{x\in A'_j,\, \mathrm{sign}(\hat{r}_n(x))=+1\}\right)\geq \frac{|A'_j|}{2}$}\\
-1 &\text{ otherwise}
\end{cases}
\]
where $\lambda$ is the Lebesgue measure. In other words, for each hypercube $A_j$, we take a majority vote, but only on $A'_j$. In particular, if $X$ is uniform on $S$ and falls in the subset $A'_j$ of a bad hypercube $A_j$, \textit{i.e.}\ such that $\tilde{c}_j\neq c_j$, then 
\[
\P\left(\left.\left| \tilde{r}_n(X)-r^{(c)}(X)\right|\geq h\left(1-2^{-\frac{1}{d}}\right)\right| X\in A'_j, \tilde{c}_j\neq c_j \right)\geq\frac{1}{2}\, .
\]
Next, \eqref{aozchaeicjoaoaijcoazijzi} gives
\[
\P\left(\left|\hat{r}_n(X)-r^{(c)}(X)\right|\geq \frac{h}{2}\left(1-2^{-\frac{1}{d}}\right)\right) \geq \P\left(\left|\tilde{r}_n(X)-r^{(c)}(X)\right|\geq h\left(1-\frac{1}{2^{\frac{1}{d}}}\right)\right)\, .
\]
For any event $B$, the independency between $\tilde{c}_j$ and $X$, and the fact that $\P(X\in A'_j)=(2K)^{-1}$ imply
\[
\P(B)\geq\sum_{j=1}^{K}\P\left(B\left|X\in A'_j, \tilde{c}_j\neq c_j\right. \right)\P(X\in A'_j)\P(\tilde{c}_j\neq c_j)
=\frac{1}{2K}\sum_{j=1}^{K}\P\left(B\left|X\in A'_j, \tilde{c}_j\neq c_j\right. \right)\P(\tilde{c}_j\neq c_j)\, .
\]
Combining those observations yields
\[
\P\left(\left|\hat{r}_n(X)-r^{(c)}(X)\right|\geq \frac{h}{2}\left(1-2^{-\frac{1}{d}}\right)\right)\geq \frac{1}{4K}\sum_{j=1}^{K}\P(\tilde{c}_j\neq c_j)\, .
\] 
We are now left to show that, for some well chosen value of $M$, we have
\[
\sup_{c\in\{-1,1\}^{K}}\frac{1}{4K}\sum_{j=1}^{K}\P(\tilde{c}_j\neq c_j)\geq \delta\, .
\]
To do so, consider a uniform random vector $(C_1,\dots,C_{K})\in\{-1,1\}^{K}$ (that is, i.i.d.\ Rademacher random variables with parameter $1/2$). Clearly,
\[
\sup_{c\in\{-1,1\}^{K}}\frac{1}{4K}\sum_{j=1}^{K}\P(\tilde{c}_j\neq c_j)\geq \frac{1}{4K}\sum_{j=1}^{K} \P(\tilde{c}_j\neq C_j)\, .
\]
Now, for each $j\in\llbracket 1,K\rrbracket$, the estimated sign $\tilde{c}_j$ might be seen as a decision rule on $C_j$, based on the data $\cD_n$. The minimal error probability is attained by the Bayes decision rule:
\[
C^\star_j:=\ind_{\P(C_j=1 | \cD_n)\geq 1/2} -\ind_{\P(C_j=1 | \cD_n) < 1/2}\, .
\]
Hence,
\[
\P(\tilde{c}_j\neq C_j)\geq \P(C^\star_j\neq C_j)=\E\left[\P(C^\star_j\neq C_j\given X_1,\dots,X_n)\right]\, .
\]
Let $X_{i_1},\dots,X_{i_\ell}$ be the variables $X_i$ that fall in the hypercube $A_j$. Conditionally on $X_1,\dots,X_n$, the Bayesian rule for $C_j$ based on $Y_1,\dots,Y_n$ only depends on $Y_{i_1},\dots,Y_{i_\ell}$, and the problem comes down to the Bayesian estimation of $C\sim \mathrm{Rad}(1/2)$ in the model $Y=Cu+W$, where $u$ is a fixed vector of $\R^\ell$ and $W$ is a centered Gaussian vector with covariance matrix $\sigma^2I_\ell$, independent of $C$. In this situation, \cite{MR1920390}, Lemma 3.2, ensures that 
\[
\P(C^\star_j\neq C_j\given X_1,\dots,X_n)=\Phi\left(-\frac{\sqrt{\sum_{s=1}^\ell g_j(X_{i_s})^2}}{\sigma}\right)=\Phi\left(-\frac{\sqrt{\sum_{i=1}^n g_j(X_{i})^2}}{\sigma}\right)\, .
\]
By Jensen's Inequality, we have
\begin{align*}
\P(C^\star_j\neq C_j)&\geq \Phi\left(-\frac{\sqrt{\sum_{i=1}^n \E\left[g_j(X_{i})^2\right]}}{\sigma}\right)= \Phi\left(-\frac{\sqrt{n\E\left[g_j(X)^2\right]}}{\sigma}\right)\, .
\end{align*}
Since, by \eqref{aozchaeicjoazi}, 
\[
\E\left[g_j(X)^2\right]=\frac{1}{A^d}\int g_j(x)^2 \d x=\frac{h^{d+2}}{2(d+1)(d+2)A^d}\, ,
\]
we are led to
\begin{align*}
\sup_{c\in\{-1,1\}^{K}}\frac{1}{4K}\sum_{j=1}^{K}\P(\tilde{c}_j\neq c_j)&\geq \frac{1}{4}\Phi\left(-\frac{1}{\sigma}\sqrt{\frac{n h^{d+2}}{2(d+1)(d+2)A^d}}\right)\, .
\end{align*}
We now use the following lower bound on the Gaussian tail (see for instance~\cite{polya1945remarks}, equation (1.5)): 
\[
\forall x\geq 0,\, \; \Phi(-x)\geq \frac{1}{2}\left(1-\sqrt{1-e^{-\frac{2}{\pi}x^2}}\right)\geq \frac{1}{4}e^{-\frac{2}{\pi}x^2}\, ,
\]
where the second inequality is by $\sqrt{1-u}\leq 1-u/2$ for $u\in [0,1]$. Hence,
\[
\sup_{c\in\{-1,1\}^{K}}\frac{1}{4K}\sum_{j=1}^{K}\P(\tilde{c}_j\neq c_j)\geq \frac{1}{16}\exp\left(-\frac{nh^{d+2}}{\pi(d+1)(d+2)\sigma^2A^d}\right)\geq \frac{1}{16}\exp\left(-\frac{nh^{d+2}}{\pi(d+1)^2\sigma^2A^d}\right)\, .
\]
The right-hand side is larger than $\delta\in ]0,1/16]$ as soon as
\[
h\leq \left(\frac{\pi(d+1)^2\sigma^2A^d}{n}\ln\left(\frac{1}{16\delta}\right)\right)^{\frac{1}{d+2}}\, .
\]
%which holds as soon as $M=A/h$ is chosen such that
%\[
%M\geq \frac{A}{\left(\frac{\pi(d+1)^2\sigma^2A^d}{n}\ln\left(\frac{1}{16\delta}\right)\right)^{\frac{1}{d+2}}}\, \cdot 
%\]
%In conclusion, recalling that $A^d=c_d \rho^{-1}$ for some constant $c_d$ depending on $d$ only, we obtain that for any $\delta\in ]0,1/16]$, for any $\rho>0$ and $\sigma>0$, for any regression estimate $\hat{r}_n$, there exists a distribution $(X,Y)\in\cF_{\rho,\sigma}$, such that, when $X\sim\mu$ independent of $\cD_n$, we have
%\[
%\P\left(\left| \hat{r}_n(X)-r(X)\right|\geq \frac{1-2^{-\frac{1}{d}}}{2}\left(\frac{c_d\pi(d+1)^2\sigma^2}{\rho n}\ln\left(\frac{1}{16\delta}\right)\right)^{\frac{1}{d+2}}\right)\leq \delta\, .
%\]
This concludes the proof of Proposition~\ref{prop:lower-bis}.
\end{proof}

\subsection{Proofs from Section~\ref{sec:preliminary}}
\label{subsec:proof-preliminary}

\begin{proof}[Proof of Lemma~\ref{lem:binomial}]
By definition of the median, we have
\[
\P\left( \left| \hat{r}^{\mathsf{mom}}_n(x)-r(x)\right|\geq t\right)\leq \P\left(\sum_{j=1}^m \ind_{\left\{\left| \hat{r}^{(j)}(x)-r(x)\right|\geq t\right\}}\geq\frac{m}{2}\right)\, .
\]
The variables $\ind_{\left\{\left| \hat{r}^{(j)}(x)-r(x)\right|\geq t\right\}}$ are i.i.d.\ Bernoulli variables with parameter $p_t(x)$. Now, if $B_1,\dots,B_m$ are i.i.d.\ Bernoulli random variables with parameter $p\in[0,1]$, then for any real number $\ell\in[0,m] $ we may write
\begin{equation}\label{piejpozpaozjd}
\P\left(\sum_{j=1}^m B_j\geq\ell\right)= \sum_{k=\lceil\ell\rceil}^{m}{ m \choose k}p^k(1-p)^{m-k}\leq p^{\lceil\ell\rceil}\sum_{k=\lceil\ell\rceil}^{m}{m \choose k}\leq p^\ell\sum_{k=0}^{m}{m \choose k}=2^m p^\ell
\, .
\end{equation}
In particular, taking $\ell=\frac{m}{2}$ gives the desired result in Lemma~\ref{lem:binomial}.
%Hence,
%\begin{align*}
%\P\left( \left| \hat{r}^{\mathsf{mom}}_n(X)-r(X)\right|\geq t\right)&\leq \E\left[ \P\left(\sum_{j=1}^m \ind_{\left\{\left| \hat{r}^{(j)}(X)-r(X)\right|\geq t\right\}}\geq\frac{m}{2}\given X\right)\right]\\
%&\leq 2^m \E\left[ p_t(X))^{m/2}\right]\, .
%\end{align*}
\end{proof}

\begin{proof}[Proof of Lemma~\ref{lem:bias-variance}]
For a given $x\in \R^d$, the difference $\hat{r}_N(x)-r(x)$ can be decomposed as
\[
\hat{r}_N(x)-r(x)=\sum_{i=1}^N W_i(x)\varepsilon_i +\sum_{i=1}^N W_i(x)\left(r(X_i)-r(x)\right)\, ,
\]
where $\varepsilon_i=Y_i-r(X_i)$. By the triangle inequality and the fact that $r$ is $1$-Lipschitz, we have
\[
\left|\sum_{i=1}^N W_i(x) r(X_i)-r(x)\right|\leq \sum_{i=1}^N W_i(x)\|X_i-x\|\, ,
\]
which establishes inequality~\eqref{eq:bias-variance1}. Next, for $t,s>0$, a union bound gives
\begin{align*}
p_{t+s}(x) &\leq \P\left( \left|\sum_{i=1}^N W_i(x)\varepsilon_i\right|\geq t\right)+\P\left(\sum_{i=1}^N W_i(x)\|X_i-x\|\geq s\right)\, ,
\end{align*}
By Markov's inequality, we have
\[
\P\left( \left|\sum_{i=1}^N W_i(x)\varepsilon_i\right|\geq t\right)\leq \frac{\E\left[\left(\sum_{i=1}^N W_i(x)\varepsilon_i\right)^2\right]}{t^2}\, ,
\]
and the assumption on the conditional variance of $\varepsilon$ implies that
\begin{align*}
\E\left[\left(\sum_{i=1}^N W_i(x)\varepsilon_i\right)^2\right]&= \sum_{i,j=1}^N \E\left[W_i(x)W_j(x)\E\left[ \varepsilon_i\varepsilon_j\given X_1,\dots,X_n\right]\right]\nonumber\\
&= \sum_{i=1}^N \E\left[W_i(x)^2\E\left[ \varepsilon_i^2\given X_i\right]\right]\nonumber\\
& \leq \sigma^2\E\left[\sum_{i=1}^N W_i(x)^2\right]\, ,
\end{align*}
which concludes the proof of~\eqref{eq:bias-variance2}.
%Hence, for a random $X$ with law $\mu$, independent of $\cD_N$, we have
%\[
%p_X(t) \leq \frac{\sigma^2 v(X)}{s^2}+ \frac{L b(X)}{t-s}\, ,
%\]
%Using that $(a+b)^{m/2}\leq 2^{m/2}(a^{m/2}+b^{m/2})$, we obtain
%\[
%\E\left[ p_X(t)^{m/2}\right]\leq \left(\frac{2}{t^2}\right)^{m/2}\left(\sigma^m \E\left[ v(X)^{m/2}\right]+L^m \E\left[b(X)^{m/2}\right]\right)\, .
%\]
%When $m\geq 2$, we may further use Jensen's Inequality and write
%\[
%\E\left[\E\left[ \sum_{i=1}^N W_i(X)\|X_i-X\|^2\given X\right]^{m/2}\right]\leq \E\left[\sum_{i=1}^N W_i(X)\|X_i-X\|^m\right]\, .
%\]
\end{proof}

\subsection{Proofs from Section~\ref{sec:nearest-neighbors}}
\label{subsec:proof-knn}

\begin{proof}[Proof of Lemma~\ref{lem:expected-distance}]
We have
\[
\E\left[ D_{(i)}(x)\right]=\int_0^D \P\left( D_{(i)}(x)>\varepsilon\right)\d\varepsilon\leq a +\int_a^D \P\left( D_{(i)}(x)>\varepsilon\right)\d\varepsilon \, ,
\]
for some $a\geq 0$ to be specified later. Observe that $D_{(i)}(x)>\varepsilon$ if and only if there are strictly less than $i$ observations in $\cB(x,\varepsilon)$. Since the number of observations in $\cB(x,\varepsilon)$ is distributed as a Binomial random variable with parameters $N$ and $\mu\left(\cB(x,\varepsilon)\right)\geq \rho \varepsilon^d$, we have
\begin{equation}\label{eq:distance-binomial}
\P\left( D_{(i)}(x)>\varepsilon\right)\leq \sum_{j=0}^{i-1} {N\choose j} (\rho\varepsilon^d)^j (1-\rho\varepsilon^d)^{N-j}\, .
\end{equation}
Applying \cite{MR2600626}, Lemma 3.1, gives, for all $p\in [0,1]$,
\[
\sum_{j=0}^{i-1} {N\choose j} p^{j} (1-p)^{N-j}\leq \frac{i}{p(N+1)}\, \cdot 
\]
Hence,
\[
\E\left[ D_{(i)}(x)\right]\leq  a +\frac{i}{N+1} \int_a^D \frac{1}{\rho\varepsilon^d} \d\varepsilon\, .
\]
For $d\geq 2$, we obtain
\[
\E\left[ D_{(i)}(x)\right]\leq  a +\frac{i}{\rho(N+1)}\cdot\frac{a^{1-d}}{d-1}\leq a\left(1+\frac{ia^{-d}}{\rho(N+1)}\right)\, .
\]
Taking $a=\left(\frac{i}{\rho(N+1)}\right)^{1/d}$, we get
\[
\E\left[ D_{(i)}(x)\right]\leq 2\left(\frac{i}{\rho(N+1)}\right)^{1/d}\, .
\]
For $d=1$, we set $a=0$ and use~\eqref{eq:distance-binomial} to deduce that
\begin{align*}
\E\left[ D_{(i)}(x)\right] &\leq \int_0^D \sum_{j=0}^{i-1} {N\choose j} (\rho\varepsilon)^j (1-\rho\varepsilon)^{N-j} \d\varepsilon\\
&= \frac{1}{\rho}\sum_{j=0}^{i-1} {N\choose j} \int_0^{\rho D} u^j(1-u)^{N-j} \d u\\
&\leq \frac{1}{\rho}\sum_{j=0}^{i-1} {N\choose j} \int_0^{1} u^j(1-u)^{N-j} \d u\, .
\end{align*}
Recognizing the Beta function $\int_0^{1} u^j(1-u)^{N-j} \d u=\frac{j!(N-j)!}{(N+1)!}$, we obtain
\[
\E\left[ D_{(i)}(x)\right]\leq \frac{1}{\rho}\sum_{j=0}^{i-1} {N\choose j} \frac{j!(N-j)!}{(N+1)!}=\frac{i}{\rho (N+1)}\, \cdot 
\]
\end{proof}

\begin{proof}[Proof of Proposition~\ref{prop:knn}]
In the $k$-\nn\ case, \eqref{eicjpaec} becomes
\[
t=2e\sigma\sqrt{\frac{2}{k}}\qquad \text{ and }\qquad s=\frac{16e^2}{k}\sum_{i=1}^k\left(\frac{i}{\rho(N+1)}\right)^{1/d}\leq 16e^2\left(\frac{km}{\rho n}\right)^{1/d}\, ,
\]
where we used that $N=\left\lfloor \frac{n}{m}\right\rfloor \geq \frac{n}{m}-1$. We thus deduce from \eqref{eicjpaecqkjx} that
\begin{equation}\label{eq:knn-general}
\P\left( \left| \hat{r}^{\mathsf{mom}}_n(x)-r(x)\right|\geq 2e\sigma\sqrt{\frac{2}{k}}+16e^2\left(\frac{km}{\rho n}\right)^{1/d}\right)\leq e^{-m}\, .
\end{equation}
When $\sigma$ and $\rho$ are known, one may then choose $k$ as the largest integer such that $\sigma\sqrt{\frac{2}{k}} \geq 8e\left(\frac{km}{\rho n}\right)^{1/d}$, i.e.
\[
k^\star=\left\lfloor \left(\frac{\sigma^2}{32e^2}\right)^{\frac{d}{d+2}}\left(\frac{\rho n}{m}\right)^{\frac{2}{d+2}}\right\rfloor\, ,
\]
which belongs to $\llbracket 1,N\rrbracket=\llbracket 1,\left\lfloor \frac{n}{m}\right\rfloor\rrbracket$ provided
\[
1\leq \left(\frac{\sigma^2}{32e^2}\right)^{\frac{d}{d+2}}\left(\frac{\rho n}{m}\right)^{\frac{2}{d+2}}\leq \frac{n}{m}\, ,
\]
i.e.
\[
\frac{m}{n}\leq \rho\left(\frac{\sigma}{4e\sqrt{2}}\right)^d \wedge \frac{32e^2}{\sigma^2\rho^{2/d}}=c\, .
\]
In this case, using that $\lfloor u\rfloor \geq u/2$ for $u\geq 1$, we get
\[
4e\sigma\sqrt{\frac{2}{k^\star}}\leq 8e\sigma\sqrt{\left(\frac{32e^2}{\sigma^2}\right)^{\frac{d}{d+2}}\left(\frac{m}{\rho n}\right)^{\frac{2}{d+2}}}\leq 32e^2\sqrt{2}\left(\frac{\sigma^2m}{\rho n}\right)^{\frac{1}{d+2}}\, .
\]
In view of~\eqref{eq:knn-general}, we have, for the optimal choice $k^\star$,
\[
\P\left( \left| \hat{r}^{\mathsf{mom}}_n(x)-r(x)\right|\geq 32e^2\sqrt{2}\left(\frac{\sigma^2m}{\rho n}\right)^{\frac{1}{d+2}}\right) \leq e^{-m}\, .
\]
Taking $\delta\in [e^{-\lfloor cn\rfloor}, 1[$ and $m=\lceil \ln(1/\delta)\rceil$, we arrive at the desired result.
\end{proof}

\begin{proof}[Proof of Proposition~\ref{prop:bagged}]
Concerning the variance term in \eqref{eicjpaec}, Proposition 2.2 in~\cite{MR2600626} and the fact that $k\leq N$ yield
\[
\sum_{i=1}^Nv_i^2\leq\frac{2k}{N}\left(1+\frac{1}{N}\right)^{2k}\leq \frac{2e^2k}{N}\leq \frac{4e^2km}{n}\, ,
\]
where for the last inequality, we used that $N=\left\lfloor \frac{n}{m}\right\rfloor\geq\frac{n}{2m}$. For the bias term in \eqref{eicjpaec}, we have to upper bound the quantity
\[
\sum_{i=1}^Nv_i\left(\frac{i}{N+1}\right)^{1/d}\, .
\]
This is the purpose of the upcoming result, whose proof is detailed just below.                                                                                                                                            
\begin{lemma}\label{lem:bagging}
For all $d\geq 1$, one has
\[
\sum_{i=1}^Nv_i\left(\frac{i}{N+1}\right)^{1/d}\leq \frac{e}{k^{1/d}}\, .
\]
\end{lemma}

In view of \eqref{eicjpaec}, we are led to
\[
t\leq 4e^2\sigma\sqrt{\frac{2km}{n}}\qquad\text{ and }\qquad s\leq \frac{16e^3 }{(\rho k)^{1/d}}\, ,
\]
and, by Lemma~\ref{lem:binomial}, for all $x\in S$,
\[
\P\left( \left| \hat{r}^{\mathsf{mom}}_n(x)-r(x)\right|\geq 4e^2\sigma\sqrt{\frac{2km}{n}}+\frac{16e^3 }{(\rho k)^{1/d}}\right)\leq e^{-m}\, .
\]
As for the $k$-\nn\ case, when $\sigma$ and $\rho$ are known, the integer $k$ may be chosen as the largest integer such that $\frac{4e}{(\rho k)^{1/d}} \geq \sigma\sqrt{\frac{2km}{n}}$, that is
\[
k^\star =\left\lfloor \left(\frac{8e^2n}{\rho^{2/d}\sigma^2m}\right)^{\frac{d}{d+2}}\right\rfloor\, ,
\]
which belongs to $\llbracket 1,N\rrbracket=\llbracket 1,\left\lfloor \frac{n}{m}\right\rfloor\rrbracket$ if
\[
\frac{m}{n}\leq \rho\left(\frac{\sigma}{2e\sqrt{2}}\right)^d \wedge \frac{8e^2}{\sigma^2\rho^{2/d}}=c\, .
\]
In this case, using that $\lfloor u\rfloor \geq u/2$ for $u\geq 1$, we get after some simplification
\[
\P\left( \left| \hat{r}^{\mathsf{mom}}_n(x)-r(x)\right|\geq 64e^3\left(\frac{\sigma^2m}{\rho n}\right)^{\frac{1}{d+2}}\right)\leq e^{-m}\, .
\]
Taking $\delta\in [e^{-\lfloor cn\rfloor},1[$ and $m=\lceil \ln(1/\delta)\rceil$ yields the desired result.
\end{proof}

\begin{proof}[Proof of Lemma~\ref{lem:bagging}]
Recall that
\[
v_i=\left(1-\frac{i-1}{N}\right)^k-\left(1-\frac{i}{N}\right)^k\, .
\]
Since $\sum_{i=1}^N v_i=1$ and since $x\mapsto x^{1/d}$ is concave on $\R_+$, Jensen's inequality gives
\begin{equation}\label{eq:bagging-jensen}
\sum_{i=1}^Nv_i\left(\frac{i}{N+1}\right)^{1/d}\leq \left(\sum_{i=1}^Nv_i\frac{i}{N+1}\right)^{1/d}\, .
\end{equation}
Now
\begin{align*}
\sum_{i=1}^N iv_i &= \sum_{i=1}^N \left\{ (i-1)\left(1-\frac{i-1}{N}\right)^k-i\left(1-\frac{i}{N}\right)^k \right\}+\sum_{i=1}^N \left(1-\frac{i-1}{N}\right)^k = \sum_{i=1}^N\left(\frac{i}{N}\right)^k\, ,
\end{align*}
and by comparing the latter to the associated integral we have
\[
\sum_{i=1}^N\left(\frac{i}{N}\right)^k\leq N\int_{1/N}^{1+1/N} u^k \d u\leq \frac{N}{k+1}\left(1+\frac{1}{N}\right)^{k+1}=\frac{N+1}{k+1}\left(1+\frac{1}{N}\right)^{k}\leq \frac{N+1}{k} e\, .
\]
Coming back to~\eqref{eq:bagging-jensen}, we get
\[
\sum_{i=1}^Nv_i\left(\frac{i}{N+1}\right)^{1/d}\leq  \left(\frac{e}{k}\right)^{1/d}\leq  \frac{e}{k^{1/d}}\, .
\]
\end{proof}

\subsection{Proofs from Section~\ref{sec:kernel-partitioning}}
\label{subsec:proof-kernel}

\begin{proof}[Proof of Proposition~\ref{prop:kernel}]
By inequality \eqref{eq:bias-variance1}, we get
\[
\left|\hat{r}_N(x)-r(x)\right|\leq \left|\sum_{i=1}^N W_i(x)\varepsilon_i\right|+\sum_{i=1}^N W_i(x)\|X_i-x\|\, .
\]
In the kernel case, we have, deterministically,
\begin{align*}
\sum_{i=1}^N W_i(x)\|X_i-x\|= \frac{1}{N_h(x)}\sum_{i=1}^N\|X_i-x\|\ind_{\|X_i-x\|\leq h}\leq h\, .
\end{align*}
Hence, for all $t>0$ and all $x\in S$, Markov's inequality yields 
\begin{align*}
p_{t+h}(x)&=\P(\left|\hat{r}_N(x)-r(x)\right|\geq t+h)\leq \P\left(\left|\sum_{i=1}^N W_i(x)\varepsilon_i\right|\geq t\right)\leq \frac{\sigma^2\E\left[ \sum_{i=1}^N W_i(x)^2\right]}{t^2}\, ,
\end{align*}
so that
\[
p_{t+h}(x)\leq\frac{\sigma^2\E\left[\frac{\ind_{N_h(x)>0}}{N_h(x)}\right]}{t^2}\, \cdot
\]
Since $N_h(x)$ is distributed as a Binomial random variable with parameters $N$ and $\mu\left(\cB(x,h)\right)$, we have, by \cite{MR1920390}, Lemma 4.1, and Assumption~\ref{aljcj}, 
%\begin{align*}
%\E\left[\frac{\ind_{B>0}}{B}\right]&=\sum_{k=1}^N \frac{1}{k}{N \choose k}p^k(1-p)^{N-k}\\
%&\leq \frac{2}{N+1}\sum_{k=1}^N {N+1 \choose k+1}p^k(1-p)^{N-k}\\
%&=\frac{2}{N+1}\sum_{k=2}^{N+1} {N+1 \choose k}p^{k-1}(1-p)^{N+1-k}\\
%&\leq \textcolor{blue}{\frac{2}{(N+1)p}}\, .
%\end{align*}
%Hence, we get  \textcolor{blue}{(pour avoir $\mu\left(\cB(x,h)\right)\geq ch^d$, il faut supposer $h\leq\rho$)}
\[
\E\left[\frac{\ind_{N_h(x)>0}}{N_h(x)}\right]\leq \frac{2}{(N+1) \mu\left(\cB(x,h)\right)}\leq \frac{2m}{\rho nh^d}\, .
\]
Hence, we obtain
\[
p_{t+h}(x)\leq \frac{2\sigma^2m}{\rho nh^dt^2}\, \cdot
\]
Since the right-hand side equals $1/4e^2$ for $t=2e\sqrt{\frac{2\sigma^2m}{\rho nh^d}}$, Lemma~\ref{lem:binomial} implies 
\[
\P\left( \left| \hat{r}^{\mathsf{mom}}_n(x)-r(x)\right|\geq 2e\sqrt{\frac{2\sigma^2m}{\rho n h^d}}+h\right)\leq e^{-m}\, .
\]
When $\sigma$ and $\rho$ are known, the bandwidth $h$ can then be optimized: taking $h$ such that $2e\sqrt{\frac{2\sigma^2m}{\rho nh^d}}=h$, i.e.
\[
h^\star=\left(\frac{8e^2\sigma^2m}{\rho n}\right)^{\frac{1}{d+2}}\, ,
\]
we see that if $m\leq \frac{\rho D^{d+2}n}{8e^2\sigma^2}$, ensuring $h^\star\leq D$, we have 
\[
\P\left( \left| \hat{r}^{\mathsf{mom}}_n(x)-r(x)\right|\geq 2 \left(\frac{8e^2\sigma^2m}{\rho n}\right)^{\frac{1}{d+2}}\right)\leq e^{-m}\, .
\]
Taking $\delta\in [e^{-\lfloor cn\rfloor},1[$ and $m=\lceil \ln(1/\delta)\rceil$ yields the desired result.
\end{proof}

\begin{proof}[Proof of Proposition~\ref{prop:partitioning}]
\textit{Mutatis mutandis}, the reasoning is the same as for kernel estimates. Here again, the bias term can indeed be deterministically bounded:
\begin{equation}\label{coazjcoiajz}
\sum_{i=1}^N W_i(x)\|X_i-x\|= \frac{1}{N_k}\sum_{i=1}^N\ind_{X_i\in A_k}\|X_i-x\|\leq \sqrt{d}K^{-1}\, .
\end{equation}
Hence, for all $t>0$ and $x\in A_k$, we are led to
\begin{equation}\label{eq:bound-variance-partitioning}
p_{t+\sqrt{d}K^{-1}}(x)\leq \frac{\sigma^2\E\left[\frac{\ind_{N_k>0}}{N_k}\right] }{t^2}\leq \frac{2\sigma^2}{(N+1)\mu(A_k)t^2} \leq \frac{2^{d+1}K^d\sigma^2m}{\rho n t^2}\, ,
\end{equation}
where we used that, if $a_k$ denotes the center of $A_k$, then by assumption~\eqref{aljcj}
\[
\mu(A_k)\geq\mu\left(\cB(a_k,(2K)^{-1})\right)\geq \rho(2K)^{-d}\, .
\]
Again, by Lemma~\ref{lem:binomial}, we get
\[
\P\left( \left| \hat{r}^{\mathsf{mom}}_n(x)-r(x)\right|\geq 2e\sqrt{\frac{2^{d+1}K^d\sigma^2m}{\rho n }}+ \sqrt{d}K^{-1}\right)\leq e^{-m}\, .
\]
One may then choose $K$ as the largest integer such that $\sqrt{d}K^{-1} \geq 2e\sqrt{\frac{2^{d+1}K^d\sigma^2m}{\rho n }}$, i.e.
\[
K_\star=\left\lfloor \left(\frac{\rho d n}{2^{d+3}e^2\sigma^2m}\right)^{\frac{1}{d+2}}\right\rfloor\, ,
\]
which belongs to $\N\setminus\{0\}$ as soon as
\[
\frac{m}{n}\leq\frac{\rho d }{2^{d+3}e^2\sigma^2} \, .
\]
Once again, using that $\lfloor u\rfloor \geq u/2$ for $u\geq 1$, we obtain
\[
2\sqrt{d}K_\star^{-1}\leq 4\sqrt{d}\left(\frac{2^{d+3}e^2\sigma^2m}{\rho d n}\right)^{\frac{1}{d+2}}\leq 16e^{2/3}\sqrt{d}\left(\frac{\sigma^2m}{\rho n}\right)^{\frac{1}{d+2}}\, ,
\]
which yields
\[
\P\left( \left| \hat{r}^{\mathsf{mom}}_n(x)-r(x)\right|\geq 16e\sqrt{d}\left(\frac{\sigma^2m}{\rho n}\right)^{\frac{1}{d+2}}\right)\leq e^{-m}\, .
\]

\end{proof}

\begin{proof}[Proof of Proposition~\ref{prop:partitioningbis}]
For all $t>0$, we have 
\[
\P\left(\sup_{x\in S}\left|\hat{r}^{\mathsf{mom}}_n(x)-r(x)\right|> t+\sqrt{d}K^{-1}\right)\leq \P\left( \sup_{x\in S}\sum_{j=1}^m \ind_{\left\{ |\hat{r}^{(j)}(x)-r(x)|> t+\sqrt{d}K^{-1}\right\}}\geq \frac{m}{2}\right)\, .
\]
Then, by inequality \eqref{eq:bias-variance1} of Lemma \ref{lem:bias-variance} and the deterministic bound \eqref{coazjcoiajz} on the bias term, it comes
\[
\P\left(\sup_{x\in S}\left|\hat{r}^{\mathsf{mom}}_n(x)-r(x)\right|> t+\sqrt{d}K^{-1}\right)\leq \P\left( \sup_{x\in S}\sum_{j=1}^m \ind_{\left\{ \left|\sum_{i=1}^N W_i^{(j)}(x)\varepsilon_i^{(j)}\right|> t\right\}}\geq \frac{m}{2}\right)\, ,
\]
where $\varepsilon_1^{(j)},\dots,\varepsilon_N^{(j)}$ stand for the noise variables in block $j$, and where
\[
W_i^{(j)}(x):=\sum_{k=1}^{K^d}\ind_{x\in A_k} \frac{1}{N_k^{(j)}}\ind_{X_i^{(j)}\in A_k}\, ,
\]
with $X_1^{(j)},\dots,X_N^{(j)}$ the features in block $j$, and $N_k^{(j)}$ the number of features in block $j$ falling into $A_k$. Noticing that $W_i^{(j)}(x)$ does not depend on the exact position of $x$ but only on the cube $A_k$ in which it lies, we obtain, with the notation $B_{i,k}^{(j)}:=\ind_{\left\{X_i^{(j)}\in A_k\right\}}$,
\begin{align*}
\P\left(\sup_{x\in S}\left|\hat{r}^{\mathsf{mom}}_n(x)-r(x)\right|> t+\sqrt{d}K^{-1}\right)&\leq \P\left( \sup_{1\leq k\leq K^d}\sum_{j=1}^m \ind_{\left\{ \left|\frac{1}{N_k^{(j)}}\sum_{i=1}^N B_{i,k}^{(j)}\varepsilon_i^{(j)}\right|> t\right\}}\geq \frac{m}{2}\right)\\
&\leq \sum_{k=1}^{K^d}\P\left(\sum_{j=1}^m \ind_{\left\{ \left|\frac{1}{N_k^{(j)}}\sum_{i=1}^N B_{i,k}^{(j)}\varepsilon_i^{(j)}\right|> t\right\}}\geq \frac{m}{2}\right)\\
&=\sum_{k=1}^{K^d}\P\left(\sum_{j=1}^m B_k^{(j)}\geq \frac{m}{2}\right)\, ,
\end{align*}  
thanks to the union bound and with the notation
\[
B_k^{(j)}:=\ind_{\left\{ \left|\frac{1}{N_k^{(j)}}\sum_{i=1}^N B_{i,k}^{(j)}\varepsilon_i^{(j)}\right|> t\right\}}.
\]
Clearly, for each $k\in\llbracket 1,K\rrbracket$, the Bernoulli random variables $(B_k^{(j)})_{1\leq j\leq m}$ are i.i.d.\ with parameter 
\[
p_k:=\P\left( \left|\frac{1}{N_k^{(j)}}\sum_{i=1}^N B_{i,k}^{(j)}\varepsilon_i^{(j)}\right|> t\right)\leq \frac{2^{d+1}K^d\sigma^2m}{\rho n t^2}\, ,
\]
where the upper bound, which does not depend on $k$, comes from \eqref{eq:bound-variance-partitioning}. We may now apply inequality \eqref{piejpozpaozjd} in the proof  of Lemma~\ref{lem:binomial} to deduce that 
\[
\P\left(\sup_{x\in S}\left|\hat{r}^{\mathsf{mom}}_n(x)-r(x)\right|> t+\sqrt{d}K^{-1}\right)\leq K^d\cdot 2^m \left(\frac{2^{d+1}K^d\sigma^2m}{\rho n t^2}\right)^{m/2}\, .
\]
Choosing $t$ appropriately, we get
\[
\P\left(\sup_{x\in S}\left|\hat{r}^{\mathsf{mom}}_n(x)-r(x)\right|> e\sqrt{\frac{2^{d+3}\sigma^2K^{d+\frac{2d}{m}}m}{\rho n}}+\sqrt{d}K^{-1}\right)\leq e^{-m}\, .
\]
The end of the proof is then the same as for Proposition~\ref{prop:partitioning}. Indeed, one may choose $K$ as the largest integer such that $\sqrt{d}K^{-1} \geq e\sqrt{\frac{2^{d+3}\sigma^2K^{d+\frac{2d}{m}}m}{\rho n}}$, i.e.
\[
K_\star=\left\lfloor \left(\frac{\rho d n}{2^{d+3}e^2\sigma^2m}\right)^{\frac{1}{d+2+\frac{2d}{m}}}\right\rfloor\, ,
\]
which belongs to $\N\setminus\{0\}$ as soon as
\[
\frac{m}{n}\leq\frac{\rho d }{2^{d+3}e^2\sigma^2} \, .
\]
Using again that $\lfloor u\rfloor \geq u/2$ for $u\geq 1$, we obtain
\[
2\sqrt{d}K_\star^{-1}\leq 4\sqrt{d}\left(\frac{2^{d+3}e^2\sigma^2m}{\rho d n}\right)^{\frac{1}{d+2+\frac{2d}{m}}}\leq 16e\sqrt{d}\left(\frac{\sigma^2m}{\rho n}\right)^{\frac{1}{d+2+\frac{2d}{m}}}\, ,
\]
which yields
\[
\P\left(\sup_{x\in S} \left| \hat{r}^{\mathsf{mom}}_n(x)-r(x)\right|\geq 16e\sqrt{d}\left(\frac{\sigma^2m}{\rho n}\right)^{\frac{1}{d+2+\frac{2d}{m}}}\right)\leq e^{-m}\, .
\]
\end{proof}

\bigskip

\noindent\textbf{Acknowledgments.} The authors would like to thank Julien Reygner for fruitful discussions as well as Matthieu Lerasle for pointing out several references and for relevant comments on a first version of the article.

\bibliographystyle{abbrvnat}
\bibliography{biblio}

\end{document}